\documentclass[smallextended,envcountsame]{svjour3} 
\smartqed
\usepackage{amsmath,amssymb}
\usepackage[dvips]{graphicx}
\usepackage{psfrag}
\usepackage{xcolor}
\usepackage{txfonts}

\journalname{Japan J.\ Indust.\ Appl.\ Math.}

\numberwithin{equation}{section}
\numberwithin{theorem}{section}\def\thetable{\thesection.\arabic{table}}
\newcommand{\R}{\mathbb R}
\newcommand{\N}{\mathbb N}
\newcommand{\bfx}{\mathbf{x}}
\newcommand{\bfy}{\mathbf{y}}
\newcommand{\bff}{\mathbf{f}}

\newcommand{\bfg}{\mathbf{g}}
\newcommand{\lam}{\lambda}
\newcommand{\I}{\mathcal{I}}
\newcommand{\Icr}{\mathcal{I}_{K}^{CR}}

\newcommand{\E}{\mathcal{E}}
\newcommand{\hv}{\hat{v}}

\newcommand{\PP}{\mathcal{P}}
\newcommand{\dd}{\mathrm{d}}
\newcommand{\hK}{\widehat{K}}
\newcommand{\tK}{\widetilde{K}_\alpha}
\newcommand{\Ka}{K_\alpha}
\newcommand{\F}{\mathbf{F}}
\newcommand{\G}{\mathbf{G}}

\begin{document}

\title{Approximating surface areas by interpolations on triangulations}

\author{Kenta Kobayashi \and Takuya Tsuchiya}

\institute{Kenta Kobayashi \at
           Graduate School of Commerce and Management, \\
           Hitotsubashi University, Japan \\
           \email{\texttt{kenta.k@r.hit-u.ac.jp}} \and
           Takuya Tsuchiya \at
           Graduate School of Science and Engineering, \\
           Ehime University, Japan \\
           \email{\texttt{tsuchiya@math.sci.ehime-u.ac.jp}}
}

\date{Received: date / Accepted: date}

\maketitle

\begin{abstract}
We consider surface area approximations by Lagrange and
Crouzeix-Raviart interpolations on triangulations.  For Lagrange
interpolation, we give an alternative proof for Young's classical result
that claims the areas of inscribed polygonal surfaces converge to the area
of the original surface under the maximum angle condition on
the triangulation.  For Crouzeix--Raviart interpolation we show that the
approximated surface areas converge to the area of the original surface
without any geometric conditions on the triangulation.
\end{abstract}

\section{Introduction}
Let $\Omega \subset \R^2$ be a bounded domain with polygonal boundary
$\partial\Omega$.  For a sufficiently smooth function defined on
$\Omega$, for example $f \in C^1(\overline{\Omega})$, the area $A(f)$ of
its graph $z = f(x,y)$ is computed (and defined) by
\[
   A(f) = \int_\Omega \sqrt{1 + |\nabla f(\bfx)|^2}\, \dd \bfx.
\]
If the smoothness assumption is weakened, however, the definition of
$A(f)$ becomes rather complicated.  (For the definition of surface area
given by Lebesgue, see Section~\ref{sec:def-sur-area}.)
The length of a curve is defined as the limit of the length of its inscribed
polygonal curves. On the contrary, the area of a surface cannot be
defined as the limit of inscribed polygonal surfaces.  In the 1880s, Schwarz
and Peano independently presented a well-known counter-example.

\begin{figure}[thb]
\begin{center}
   \psfrag{A}[][]{$2\pi r$}
   \psfrag{B}[][]{$H$}
  \includegraphics[width=4.5cm]{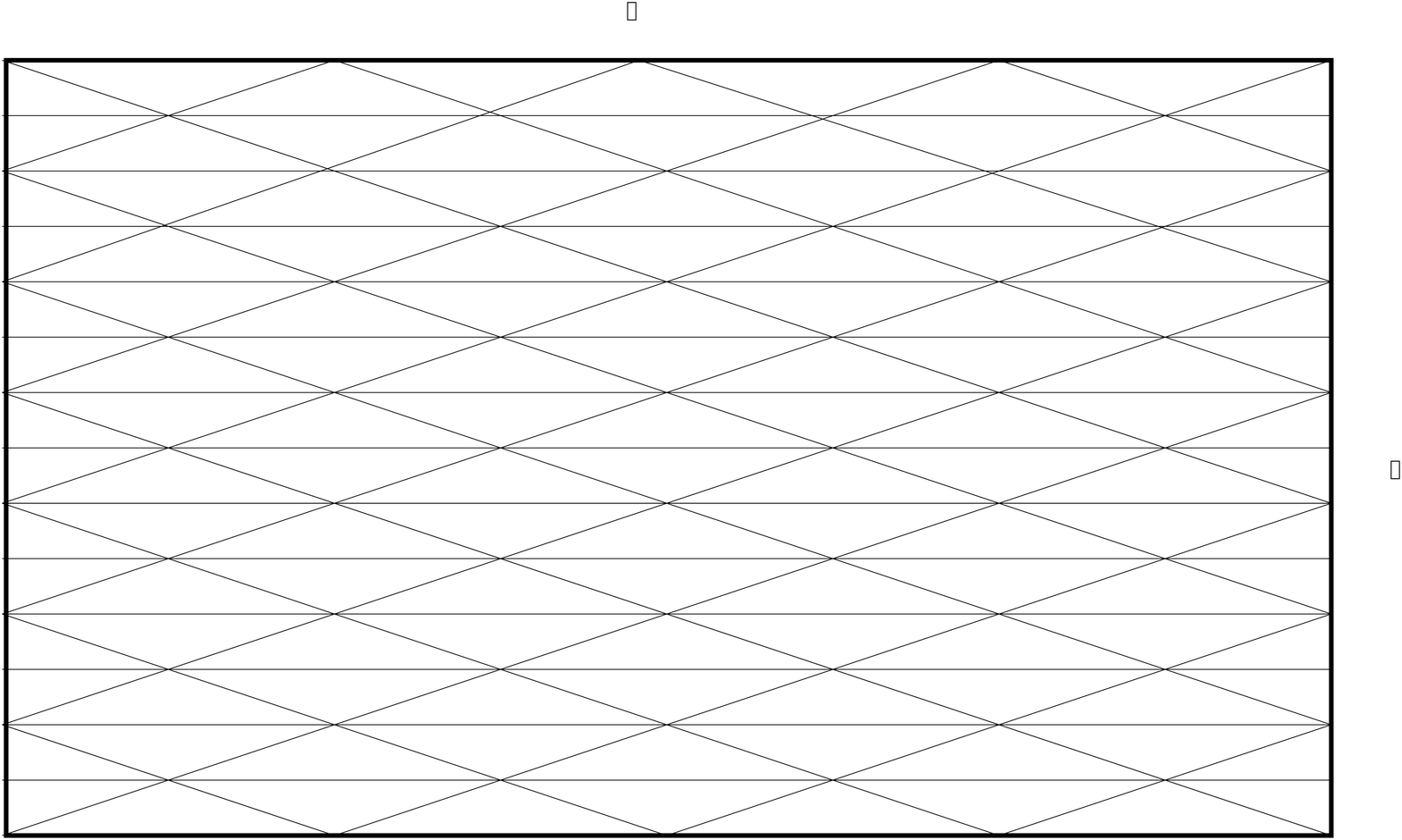}
  \hspace{0.5cm}
  \includegraphics[width=3cm,angle=90]{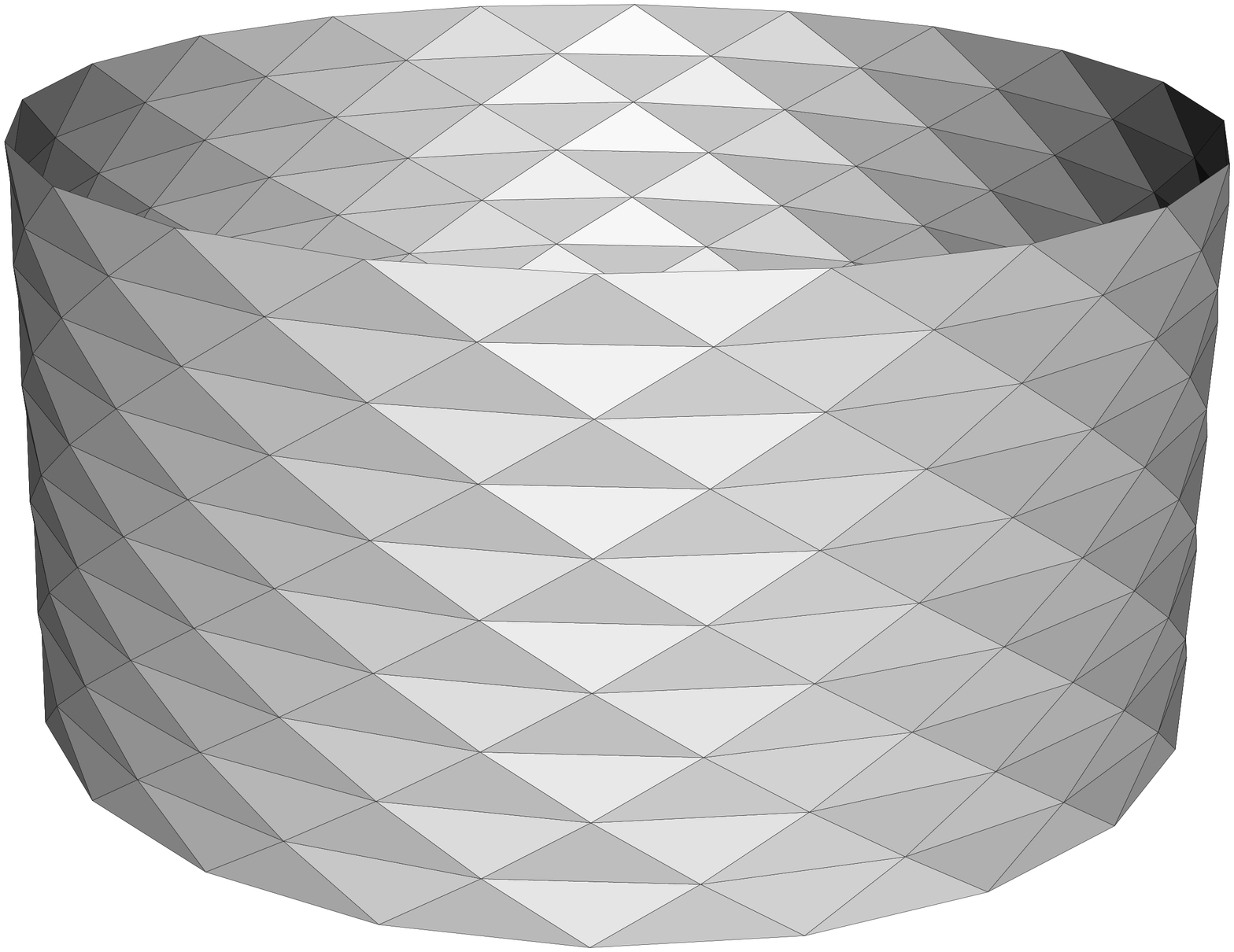} 
 \caption{Schwarz--Peano's example, called ``Schwarz's lantern''.}
 \label{fig1}
\end{center} 
\end{figure}

Let $\Omega$ be a rectangle of height $H$ and width $2\pi r$.
Let $m$, $n$ be positive integers.  Suppose that this rectangle
is divided into $m$ equal strips, each of height $H/m$.
Each strip is then divided into isosceles triangles whose
base length is $2\pi r/n$, as depicted in Figure~\ref{fig1}.
Then, the piecewise linear map $\varphi_\tau:\Omega \to \R^3$
is defined by ``rolling up this rectangle'' so that
all vertices are on the cylinder of height $H$ and radius $r$.
The cylinder is then approximated by the inscribed polygonal surface,
which consists of $2mn$ congruent isosceles triangles.
Because the height of each triangle is
$\sqrt{(H/m)^2 + r^2(1 - \cos(\pi/n))^2}$
and the base length is $2r\sin(\pi/n)$, 
the area $A_E$ of the inscribed polygonal surface
\footnote{The sum of areas of triangles. The subscript `$E$' of $A_E$
stands for `Elementary'.} is
\begin{align*}
  A_E & = 2 m n r \sin \frac{\pi}{n} 
  \sqrt{\left(\frac{H}{m}\right)^2 + r^2 
        \left(1 - \cos \frac{\pi}{n}\right)^2} \\
   & = 2 \pi r \frac{\sin \frac{\pi}{n}}{\frac{\pi}{n}}
  \sqrt{H^2 + \frac{\pi^4r^2}{4} \left(\frac{m}{n^2}\right)^2
   \left(\frac{\sin \frac{\pi}{2n}}{\frac{\pi}{2n}}\right)^4 }.
\end{align*}
If $m,n \to \infty$, we observe
\begin{align*}
& \lim_{m,n \to \infty} A_E = 2\pi r \sqrt{H^2 + \frac{\pi^4 r^2}4 
    \lim_{m,n \to \infty} \left(\frac{m}{n^2}\right)^2},
\end{align*}
and in particular, 
\begin {align*}
  \lim_{m,n \to \infty} A_E = 2\pi r H  \;\; \text{ if and only if } 
  \lim_{m,n \to \infty} \frac{m}{n^2} = 0.
\end{align*}

The example given by Schwarz and Peano has convinced mathematicians of
the need to impose some geometric assumption on such triangulations to
approximate the surface area by Lagrange interpolation.  The known
geometric conditions on triangulations are as follows:  Let
$\{\tau_k\}_{k=1}^\infty$ be a sequence of triangulations of $\Omega$
such that $\lim_{k\to \infty}|\tau_k|=0$, where
$|\tau_k| := \max_{K\in \tau_k}\mathrm{diam}K$.  For a given continuous
function $f \in C(\overline{\Omega})$, its Lagrange interpolation on the
triangulation $\tau_k$ is denoted by $\I_{\tau_k}^L f$.
We denote by $A_L(f)$ the surface area of the
graph $z = f(x,y)$ in the sense of Lebesgue.  
We also denote by $A_E(\I_{\tau_k}^L f)$ the surface area of
the Lagrange interpolation $\I_{\tau_k}^L f$.

\vspace{0.3cm}
\textbf{Minimum angle condition.}  Let $\theta_K^m$ be
the minimum inner angle of a triangle $K \in \tau_k$.  Suppose that 
there exists a contant $\theta_1$, $0 < \theta_1 \le \pi/3$, such that
\begin{align*}
     \theta_1 \le \theta_K^m, \quad \forall K \in \tau_k, \quad
     k = 1,2,\cdots.
\end{align*}
Then, $\{\tau_k\}$ is said to satisfy the \textit{minimum angle condition}. 
Rademacher showed \cite{Rademacher1}, \cite{Rademacher2} that if $\{\tau_k\}$
satisfies the minimum angle condition, then, for
$f \in W^{1,\infty}(\Omega)$, we have
\begin{align}
    \lim_{k \to \infty} A_E(\I_{\tau_k}^L f) = A_L(f).    
      \label{area-convergence}
\end{align}

\vspace{0.3cm}
\textbf{Maximum angle condition.}  Let $\theta_K^M$ be
the maximum inner angle of a triangle $K \in \tau_k$.  Suppose that 
there exists a constant $\theta_2$, $\pi/3 \le \theta_2 < \pi$, such that
\begin{align*}
     \theta_K^M \le \theta_2, \quad \forall K \in \tau_k, \quad
     k = 1,2,\cdots.
\end{align*}
Then, $\{\tau_k\}$ is said to satisfy the \textit{maximum angle condition}. 
Young showed \cite{Young} that if $\{\tau_k\}$ satisfies the maximum
angle condition, then we have \eqref{area-convergence}
for $f \in W^{1,\infty}(\Omega)$.

\vspace{0.3cm}
Note that the minimum and maximum angle conditions were rediscovered by
researchers of finite element methods some 50 years after Rademacher and
Young \cite{KobayashiTsuchiya2}.  For the above mentioned results,
readers are referred to \cite{Cesari}, \cite{Rado1}, \cite{Rado2}.

Recently, the authors presented the following result.

\vspace{0.2cm}
\textbf{Circumradius condition.}  Let $R_K$ be the circumradius
of the triangle $K \in \tau_k$.  Suppose that
\begin{align*}
     \lim_{k \to \infty} \max_{K \in \tau_k} R_K = 0.
\end{align*}
Then, $\{\tau_k\}$ is said to satisfy the \textit{circumradius condition}. 

\vspace{0.2cm}
Let $R_{m,n}$ be the circumradius of the triangles in
Schwarz's lantern.  It has been shown in \cite{KobayashiTsuchiya2} that
\begin {align*}
  \lim_{m,n \to \infty} A_E = 2\pi r H  \;\; \text{ if and only if } 
  \lim_{m,n \to \infty} R_{m,n} = 0,
\end{align*}
and \eqref{area-convergence} was proved under the circumradius condition
for $f \in W^{2,1}(\Omega)$.  From these facts, we can infer that the
circumradius condition is the best possible geometric condition
of triangulations to assure the convergence in \eqref{area-convergence}.
 
One of the aims of this paper is to give an alternate proof of Young's
result using the modern theory of finite element methods.  That is, 
in Section~\ref{sec:Lagrange}, we
will show \eqref{area-convergence} for $f \in W^{1,\infty}(\Omega)$
under the maximum angle condition using the results given in
\cite{KobayashiTsuchiya3}.

Crouzeix--Raviart interpolation is defined using integrals of
the given function on the edges of triangles.
The other, more important aim of this paper is to show that the surface area
$A_L(f)$ is approximated by the Crouzeix--Raviart interpolation
$I_{\tau_k}^{CR}f$ without any geometric conditions on the triangulation.
To this end, we develop the error analysis of Crouzeix--Raviart
interpolation in Section~\ref{sec:CrouzeixRviart}. 
Using the error analysis of Crouzeix--Raviart interpolation,
the main theorem (Theorem~\ref{thm:main}) of this paper is stated and
proved in Section~\ref{main-section}.  In Section~\ref{sec:parametric},
we will show that the results obtained in
Sections~\ref{sec:Lagrange} and \ref{main-section} for the graphs
of functions on $\Omega$ hold for parametric surfaces.
Finally, in Section~\ref{numericalexpriments}, we present the results of
numerical experiments to confirm the theoretical results.
We also mention some concluding remarks regarding further research.

\section{Preliminaries}
\subsection{Notation and the basic definitions}
Let $\R^d$ be the $d$-dimensional Euclidean space.
We denote the Euclidean norm of $\bfx \in \R^d$ by
$|\bfx|$.  Let $\R^{d*} := \{l:\R^d \to \R : l \text{ is linear}\}$
be the dual space of $\R^d$.  We always regard $\bfx \in \R^d$ as a
column vector and $\mathbf{a} \in \R^{d*}$ as a row vector.
For a matrix $A$ and $\bfx \in \R^d$, $A^\top$ and $\bfx^\top$
denote their transpositions.
For a differentiable function $f$ with $d$ variables,
its gradient $\nabla f = \mathrm{grad} f \in \R^{d*}$ is the row vector
\[
  \nabla f = \nabla_\bfx f := 
  \left(\frac{\partial f}{\partial x_1}, ..., 
    \frac{\partial f}{\partial x_d}\right), \qquad
    \bfx := (x_1, ..., x_d)^\top.
\]

Let $\N_{0}$ be the set of nonnegative integers.
For $\delta = (\delta_1,...,\delta_d) \in (\N_{0})^d$,
the multi-index $\partial^\delta$ of partial differentiation 
(in the sense of distribution) is defined by
\[
    \partial^\delta = \partial_\bfx^\delta
    := \frac{\partial^{|\delta|}\ }
   {\partial x_1^{\delta_1}...\partial x_d^{\delta_d}}, \qquad
   |\delta| := \delta_1 + ... + \delta_d.
\]
If $d=2$, we use the notation $f_x$ and $f_y$ instead of 
$\partial f/\partial x$ and $\partial f/\partial y$, respectively.

Let $\Omega \subset \R^2$ be a (bounded) domain.  The usual Lebesgue
space is denoted by $L^p(\Omega)$ for $1 \le p \le \infty$.
For a positive integer $k$, the Sobolev space $W^{k,p}(\Omega)$ is
defined by
$\displaystyle
  W^{k,p}(\Omega) := 
  \left\{v \in L^p(\Omega) \, | \, \partial^\delta v \in L^p(\Omega), \,
   |\delta| \le k\right\}$.
The norm and semi-norm of $W^{k,p}(\Omega)$ are defined,
for $1 \le p < \infty$, by
\begin{gather*}
  |v|_{k,p,\Omega} := 
  \biggl(\sum_{|\delta|=k} |\partial^\delta v|_{0,p,\Omega}^p
   \biggr)^{1/p}, \quad   \|v\|_{k,p,\Omega} := 
  \biggl(\sum_{0 \le m \le k} |v|_{m,p,\Omega}^p \biggr)^{1/p},
\end{gather*}
and $\displaystyle   |v|_{k,\infty,\Omega} := 
  \max_{|\delta|=k} \left\{\mathrm{ess}
   \sup_{\hspace{-5mm}\bfx \in\Omega}|\partial^\delta v(\bfx)|\right\}$,
 $\displaystyle   \|v\|_{k,\infty,\Omega} := 
  \max_{0 \le m  \le k} \left\{|v|_{m,\infty,\Omega}\right\}$.

Let $\bff:\Omega \to \R^d$ with $\bff = (f^1,\cdots,f^d)$.  If
$f^i \in W^{k,p}(\Omega)$, $i=1,\cdots,d$, we write $\bff$ as
$\bff \in W^{k,p}(\Omega;\R^d)$.  Their norms are defined similarly.

\subsection{Triangulation of bounded polygonal domains and Lagrange
 and Crouzeix--Raviart interpolations}
Throughout this paper, $K$ is a triangle in $\R^2$.  Let
$\Omega \subset \R^2$ be a bounded polygonal domain.  A triangulation
$\tau$ of $\Omega$ is a set of triangles that satisfies the following
properties.
\begin{itemize}
 \item $\displaystyle \overline\Omega = \bigcup_{K\in\tau} K$.
 \item If $K_1$, $K_2 \in \tau$, we have either
 $K_1 \cap K_2 = \emptyset$, or $K_1 \cap K_2$ is a vertex or
an edge of both $K_1$ and $K_2$.
\end{itemize}
Because of the second property, the triangulations discussed here are
sometimes called \textit{face-to-face} triangulations.
For a triangulation $\tau$, the fineness $|\tau|$ is defined by
\[
     |\tau| := \max_{K \in \tau} \mathrm{diam}K.
\]
We denote by $\PP_1$ the set of all polynomials with two variables whose
orders are at most $1$.  For a triangulation $\tau$ of $\Omega$, we
define the set $S_{\tau}$ of all piecewise linear continuous functions by
\[
    S_\tau := \left\{f \in C^0(\overline\Omega) \bigm| f|_K \in \PP_1, 
     \forall K \in\tau \right\}.
\]

Let $\bfx_i$, $i=1,2,3$ be vertices of a
triangle $K$.  Let $e_i$ be the edge of $K$ opposite to $\bfx_i$.  For a
continuous function $f \in C(K)$, the Lagrange interpolation
$\I_K^L f \in \PP_1$ on $K$ is defined by
$f(\bfx_i) = (\I_K^L f)(\bfx_i)$, $i=1,2,3$.  It is clear that,
for $f \in C(\overline{\Omega})$ and a triangulation $\tau$ of
$\Omega$, we can define the Lagrange interpolation
$\I_\tau^L f \in S_\tau$ as
\begin{equation*}
    \I_\tau^L f\big|_K = \I_K^L f, \qquad \forall K \in \tau.
\end{equation*}

Next, let the polynomial $\theta_i \in \PP_1$, $i=1,2,3$ be
defined by
\begin{equation*}
   \int_{e_i} \theta_i(\bfx) \dd s = 1, \qquad
   \int_{e_i} \theta_j(\bfx) \dd s = 0, \quad i \neq j.
\end{equation*}
Using the barycentric coordinate $\lam_i(\bfx)$ on $K$,
this can be written as
\[
    \theta_i(\bfx) := \frac{1}{|e_i|}(1 - 2 \lam_i(\bfx)).
\]
For a function $v \in W^{1,1}(K)$ on $K$, the (non-conforming)
Crouzeix--Raviart interpolation $\Icr v$ is defined by

\begin{equation*}
   \Icr v := \sum_{i=1}^3 
   \left( \int_{e_i} v\, \dd s\right) \theta_i.
\end{equation*}
Note that $\Icr v$ is well-defined because the trace operator
$\gamma_i: W^{1,1}(K) \to L^1(e_i)$ is continuous.
Moreover, we have
\begin{equation*}
   \int_{e_i}\left(v - \Icr v\right) \dd s = 0, \qquad
    i = 1, 2, 3.
\end{equation*}
The  Crouzeix--Raviart interpolation $\Icr v \in \PP_1$
may be defined using this equality.  The global (non-conforming)
Crouzeix--Raviart interpolation $\I_\tau^{CR}f \in L^\infty(\Omega)$ on
$\tau$ is defined by
\begin{equation*}
    \I_\tau^{CR} f\big|_K = \I_K^{CR} f, \qquad \forall K \in \tau.
\end{equation*}
Note that $\I_\tau^{CR} f$ is not continuous in general.  Let $K_1$,
$K_2 \in \tau$ be two adjacent triangles in $\tau$.  Then, on
$e = K_1 \cap K_2$, $\I_\tau^{CR} f$ is continuous only at the
midpoint of $e$.  In Figure~\ref{fig2}, we show the graphs of Lagrange
and Crouzeix-Raviart interpolations of the function
$f(x,y) = \sqrt{a^2 - x^2}$, $a = 1.1$ on a triangulation on
$\Omega:= (-1,1)\times(-1,1)$, similar to the one depicted in
Figure~\ref{fig1}.

\vspace{-0.1cm}
\begin{figure}[thbp]
\begin{center}
  \includegraphics[width=5.6cm]{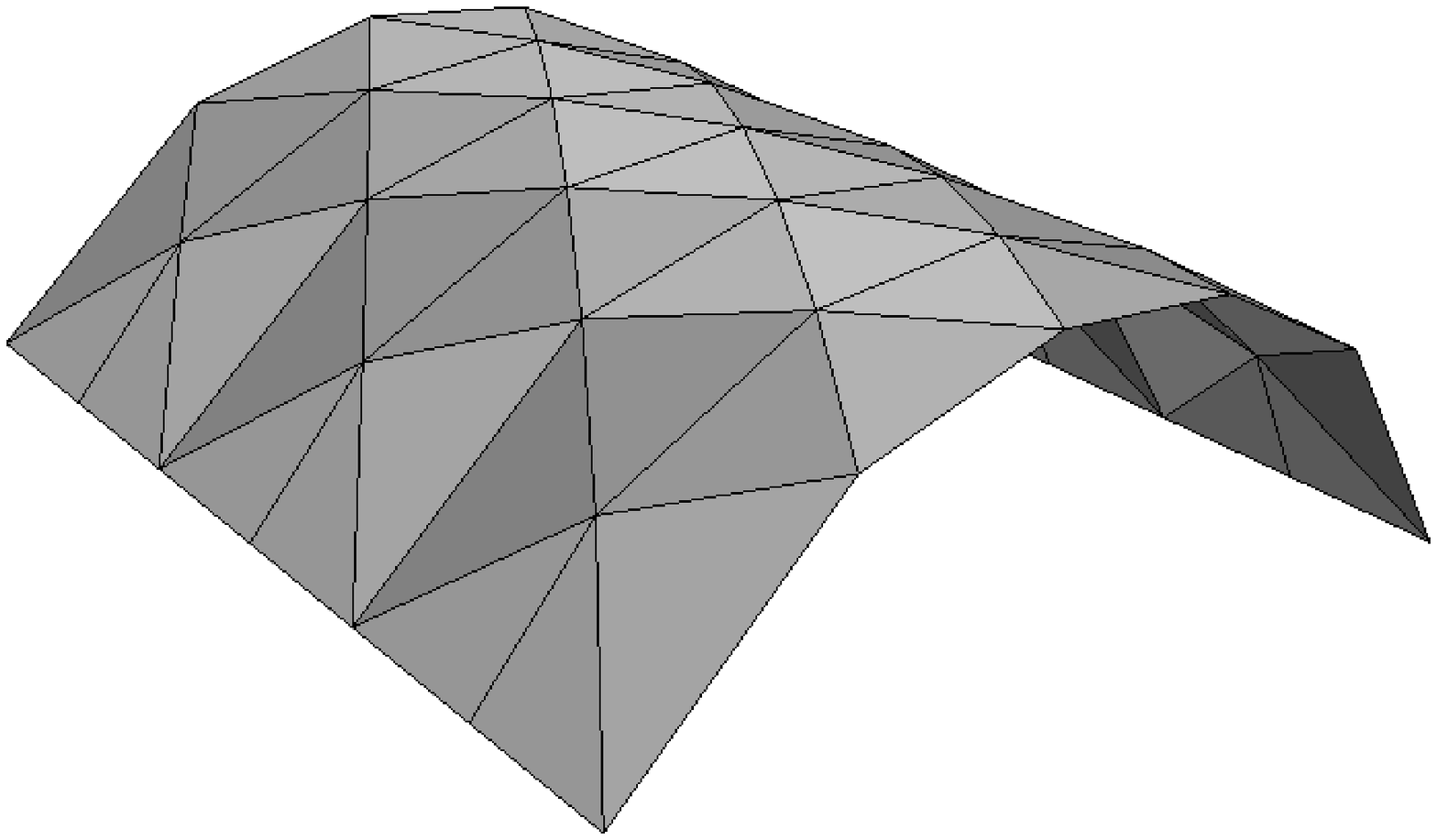}
  \hspace{0.5cm}
  \includegraphics[width=5.6cm]{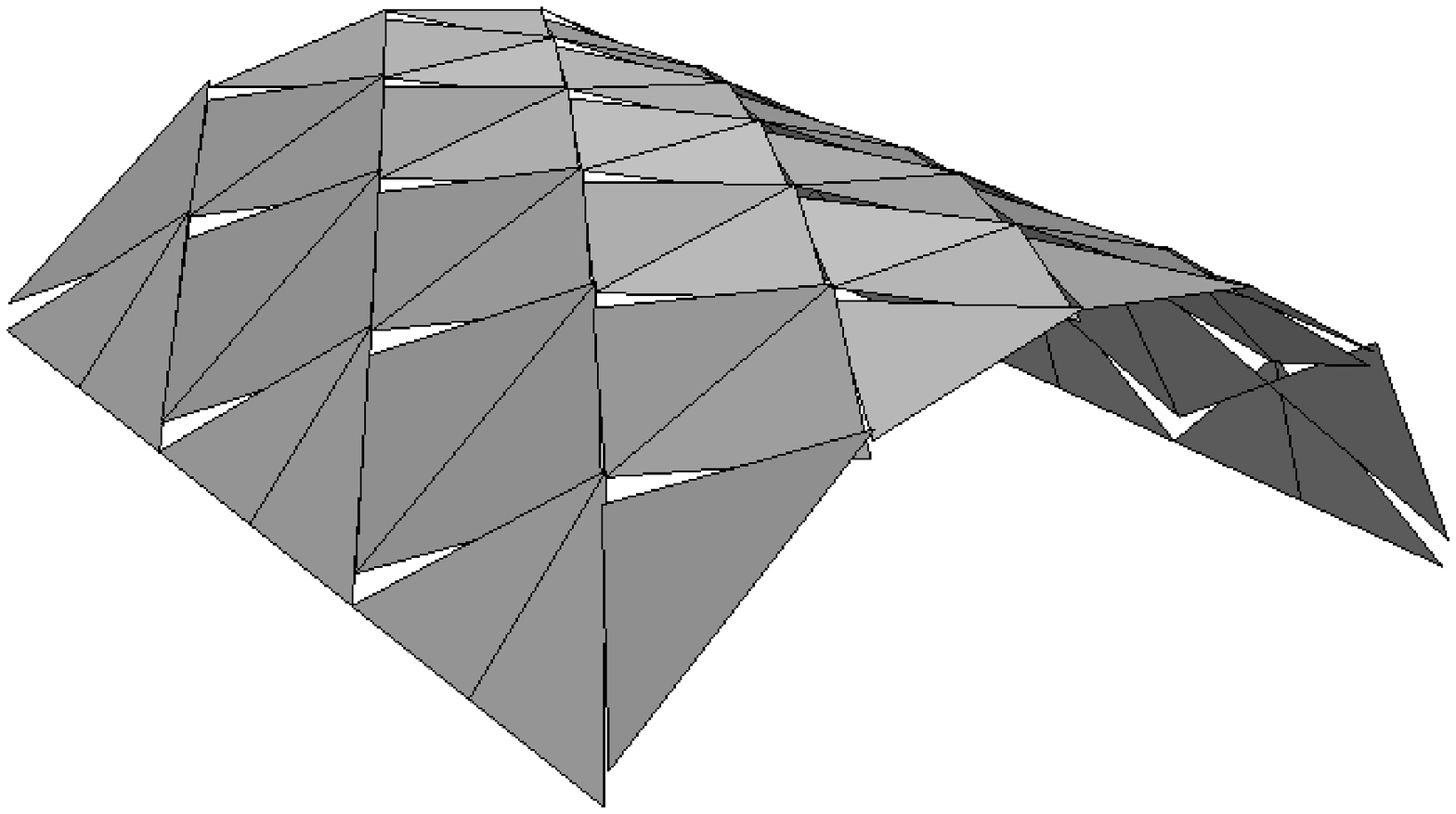} 
 \caption{Lagrange and Crouzeix--Raviart interpolations.}
 \label{fig2}
\end{center} 
\end{figure}

\vspace{-0.1cm}
  For the definitions of Lagrange and Crouzeix--Raviart
interpolations, readers are referred to textbooks on finite
element methods, such as \cite{BrennerScott}, \cite{Ciarlet},
and \cite{ErnGuermond}.

\subsection{Lebesgue's definition of the surface area and Tonelli's theorem}
\label{sec:def-sur-area}
At present, the most general definition of surface area is 
that of Lebesgue.  Let $\Omega:=(a,b)\times (c,d) \subset \R^2$ be a
rectangle and $\tau_n$ be a sequence of triangulations of $\Omega$ such
that $\lim_{n \to \infty}|\tau_n| = 0$.  Let
$f \in C^0(\overline{\Omega})$ be a given continuous function.
Let $f_n\in S_{\tau_n}$ be such that $\{f_n\}_{n=1}^\infty$ converges
uniformly to $f$ on $\overline{\Omega}$.  Note that the graph of
$z = f_n(x,y)$ is a set of triangles, and its area is defined as a sum of
these triangular areas.  We denote this area by $A_E(f_n)$, and have
\[
   A_E(f_n) = \int_{\Omega} \sqrt{1 + |\nabla f_n|^2} \dd x.
\]
Let $\Phi_f$ be the set of all such sequences
 $\{(f_n,\tau_n)\}_{n=1}^\infty$.
Then, the area $A_L(f) = A_L(f;\Omega)$ of the graph $z=f(x,y)$
is defined by
\[
   A_L(f) = A_L(f;\Omega):= \inf_{\{(f_n,\tau_n)\} \in \Phi_f} 
  \liminf_{n\to \infty}A_E(f_n).
\]
This $A_L(f)$ is called the \textbf{surface area of} $z = f(x,y)$ 
\textbf{in the Lebesgue sense}.
For a fixed $f$, $A_L(f;\Omega)$ is additive and continuous
with respect to the domain $\Omega$.
Tonelli presented the following theorem.

For a continuous function $f\in C^0(\overline{\Omega})$, we define
$W_1(x)$, $W_2(y)$ by
\begin{align*}
 W_1(x) := \sup_{\tau(y)} \sum_{i} |f(x,y_{i-1}) - f(x,y_i)|,
   \quad x \in (a,b),\\
 W_2(y) := \sup_{\tau(x)} \sum_{j} |f(x_{j-1},y) - f(x_j,y)|,
   \quad y \in (c,d),
\end{align*}
where $\tau(y)$, $\tau(x)$ are the subdivisions
 $c = y_0 < y_1 < \cdots < y_N = d$ and
 $a = x_0 < x_1 < \cdots < x_M = b$, respectively, and `$\sup$' is
taken for all such subdivisions.  Then, a function $f$ has 
\textbf{bounded variation in the Tonelli sense} if
\[
   \int_a^b W_1(x) \dd x + \int_c^d W_2(y) \dd y < \infty.
\]
Additionally, a function $f$ is said to be
\textbf{absolutely continuous in the Tonelli sense} if, 
for almost all $y \in (c,d)$ and $x \in (a,b)$,
the functions $g(x) := f(x,y)$ and $h(y) := f(x,y)$
are absolutely continuous on $(a,b)$ and $(c,d)$, respectively.
The following theorem is well-known.

\begin{theorem}[Tonelli] \label{Tonelli}
For a continuous function $f \in C(\overline{\Omega})$ defined on a
rectangular domain $\Omega$, its graph $z = f(x,y)$ has
finite area $A_L(f) < \infty$ if and only if $f$ has bounded variation
in the Tonelli sense. If this is the case, we have
\begin{equation}
   A_L(f) \ge \int_\Omega \sqrt{1 + |\nabla f(\bfx)|^2}\, \dd \bfx.
  \label{tonelli}
\end{equation}
In the above inequality, the equality holds if and only if
$f$ is absolutely continuous in the Tonelli sense.
\end{theorem}

For a proof of this theorem, see \cite[Chapter V, pp.163--185]{Saks}.
It follows from Tonelli's theorem that if $f \in W^{1,\infty}(\Omega)$,
then the area $A_L(f)$ is finite and the equality holds in
\eqref{tonelli}. 

\subsection{Affine linear transformation of triangles}
Let $\hK$ be the reference triangle with vertices 
$\hat\bfx_1 = (0,0)^\top$, $\hat\bfx_2 = (1,0)^\top$,  and 
$\hat\bfx_3 = (0,1)^\top$.  For $\alpha$, $0 < \alpha \le 1$,
let $\Ka$ and $\tK$ be the triangles with vertices $(0,0)^\top$, $(1,0)^\top$,
$(0,\alpha)^\top$, and  $\bfx_1 = (0,0)^\top$, $\bfx_2 = (1,0)^\top$,
$\bfx_3 = (\alpha s,\alpha t)^\top$, respectively, where $s^2 + t^2 = 1$,
$t > 0$. Without loss of
generality, we may assume that $e_1$ is the longest edge of $\tK$.
Let $\theta$ be the angle between $e_2$ and $e_3$ in $\tK$.
Then, $s = \cos\theta$, $t = \sin\theta$, and the assumption
that $e_1$ is the longest yields
\begin{equation*}
   s= \cos\theta \le \frac{\alpha}{2} \le \frac{1}{2}, \qquad
    \frac{\pi}{3} \le \theta < \pi.
\end{equation*}
Note that an arbitrary triangle in $\R^2$ can be transformed to $\tK$ by a
sequence of scaling, translation, rotation, and mirror imaging.

\begin{figure}[thb]
\begin{center}
  \psfrag{a}[][]{$e_3$}
  \psfrag{b}[][]{$e_2$}
  \psfrag{c}[][]{$e_1$}
  \psfrag{t}[][]{$\theta$}
  \psfrag{x1}[][]{$\bfx_1$}
  \psfrag{x2}[][]{$\bfx_2$}
  \psfrag{x3}[][]{$\bfx_3$}
  \psfrag{K}[][]{$\tK$}
  \includegraphics[width=6cm]{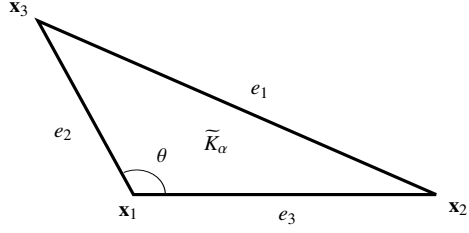}
 \caption{The triangle under consideration.
 The vertices are
 $\bfx_1=(0,0)^\top$, $\bfx_2=(1,0)^\top$, and
 $\bfx_3=(\alpha s,\alpha t)^\top$,
 where $s^2 + t^2 = 1$, $t > 0$, and $0 < \alpha \le 1$.
 We assume that $\alpha = |e_2| \le |e_3| = 1 \le |e_1|$.}
 \label{LiuKikuchiTriangle}
\end{center}
\end{figure}

We define the $2 \times 2$ matrices as 
\begin{align}
  A := \begin{pmatrix}
	   1 &  s \\ 0 &  t 
	 \end{pmatrix}, \qquad
    B := A^{-1} = \begin{pmatrix}
	   1 & - st^{-1} \\ 0 & t^{-1}
	 \end{pmatrix}.
  \label{matrixA}
\end{align}
Then, $\Ka$ can be transformed to $\tK$ by the transformation
$\bfy = A\bfx$.  Moreover, a function $v(\bfy) \in W^{1,p}(\tK)$ is 
pulled-back to the function $\hv(\bfx) \in W^{1,p}(\Ka)$ as
$\hv(\bfx) := v(A\bfx) = v(\bfy)$.  Then, we have
$\nabla_\bfx \hv = (\nabla_\bfy v) A$,
$\nabla_\bfy v = (\nabla_\bfx \hv) B$, and
$|\nabla_\bfy v|^2 = |(\nabla_\bfx \hv) B|^2
= (\nabla_\bfx \hv) BB^{\top} (\nabla_\bfx \hv)^\top$.
A simple computation yields that $A^\top A$ has eigenvalues $1 \pm |s|$, and
$BB^{\top}$ has eigenvalues
$1/(1 \pm |s|) = (1\mp|s|)/t^2$.  Hence, we have
\begin{gather}
 \frac{1-|s|}{t^2} |\nabla_\bfx \hv|^2 \le |\nabla_\bfy v|^2
   \le \frac{1+|s|}{t^2} |\nabla_\bfx \hv|^2.
 \label{eigen-est}
\end{gather}

Note that, for $N$ positive real numbers $U_1, ..., U_N$, the following
inequalities hold:
\begin{gather}
   \sum_{k=1}^N U_k^p \le N^{\tau(p)}
   \left(\sum_{k=1}^N U_k^2\right)^{p/2}, \quad
   \tau(p) := \begin{cases}
         1-p/2, & 1 \le p \le 2 \\
         0,     & 2 \le p < \infty
     \end{cases},  \label{tau}\\
    \left(\sum_{k=1}^N U_k^2\right)^{p/2} \le N^{\gamma(p)}
   \sum_{k=1}^N U_k^p, \quad
   \gamma(p) := \begin{cases}
         0, &  1 \le p \le 2 \\
         p/2 - 1,     & 2 \le p < \infty
     \end{cases}.
    \label{gamma}
\end{gather}
Combining \eqref{eigen-est} with \eqref{tau}, \eqref{gamma}, and
noting that the determinant of $A$ is $t$, we have, for $1 \le p < \infty$,
{\allowdisplaybreaks
\begin{align*}
  |v|_{1,p,\tK}^p & = \int_{\tK} \sum_{|\delta| = 1}
   |\partial_\bfy^\delta v(\bfy)|^p \dd \bfy 
   \ge 2^{-\gamma(p)} \int_{\tK} \left(
       |\nabla_\bfy v(\bfy)|^2 \right)^{p/2} \dd \bfy \\
   & \ge 2^{-\gamma(p)} \left(\frac{1-|s|}{t^2}\right)^{p/2}
      \int_{\tK} \left(|\nabla_\bfx\hv(\bfx)|^2 \right)^{p/2} \dd \bfy \\
   & = 2^{-\gamma(p)}  \left(\frac{1-|s|}{t^2}\right)^{p/2}t
      \int_{K_\alpha} \left(|\nabla_\bfx\hv(\bfx)|^2 \right)^{p/2}
     \dd \bfx \\
   & \ge 2^{-(\tau(p)+\gamma(p))} 
   \left(\frac{1-|s|}{t^2}\right)^{p/2}t
      \int_{K_\alpha} 
      \sum_{|\delta|=1} |\partial_\bfx^\delta \hv(\bfx)|^p  \dd \bfx \\
    & = 2^{-(\tau(p)+\gamma(p))} 
   \left(\frac{1-|s|}{t^2}\right)^{p/2}t |\hv|_{1,p,K_\alpha}^p,
\end{align*}
}
and similarly, 
\begin{align*}
  |v|_{1,p,\tK}^p \le2^{\tau(p)+\gamma(p)} 
   \left(\frac{1+|s|}{t^2}\right)^{p/2}t |\hv|_{1,p,K_\alpha}^p.
\end{align*}

Let $K$ be an arbitrary triangle and $K_1$ be the right triangle
obtained by a composition of parallel translation, mirror
imaging, and $A^{-1}$.  As before, any $v \in W^{1,p}(K)$ may be
pulled-back to the function $\hv := v\circ \rho \in W^{1,p}(K_1)$.
Then, in exactly the same manner, we obtain
\begin{align*}
  2^{-\eta(p)}\frac{(1-|s|)^{1/2}}{t^{1-1/p}} |\hv|_{1,p,K_1} \le
  |v|_{1,p,K} \le 2^{\eta(p)} \frac{(1+|s|)^{1/2}}{t^{1-1/p}}
   |\hv|_{1,p,K_1},
\end{align*}
where
$\eta(p):= 1/p - 1/2$ for $1 \le p \le 2$ and
$\eta(p):= 1/2- 1/p$ for $2 \le p < \infty$.
By letting $p \to \infty$, we also obtain
\begin{align}
  \frac{(1-|s|)^{1/2}}{\sqrt{2}\, t} |\hv|_{1,\infty,K_1} \le
  |v|_{1,\infty,K} \le \frac{\sqrt{2} (1+|s|)^{1/2}}{t}
   |\hv|_{1,\infty,K_1}.
  \label{transinf}
\end{align}

\section{Approximating the surface area by Lagrange interpolation}
\label{sec:Lagrange}
Let $K_1$ be a right triangle whose vertices are 
$\hat\bfx_1 := (0,0)^\top$, $\hat\bfx_2 := (h_1,0)^\top$, and
$\hat\bfx_3 := (0,h_2)^\top$, where $0 < h_2 \le h_1$.  Let $K$ be the
triangle whose vertices are defined by $\bfx_i:= A\hat\bfx_i$,
$i=1,2,3$, where the matrix $A$ is defined by \eqref{matrixA}.
Without loss of generality, we may assume that the angle at
the vertex $A\bfx_1$ is the maximum angle of $K$.  Note that
an arbitrary triangle is obtained from $K$ by a combination of rotation,
translation, and mirror imaging.

As before, an arbitrary function $\hv \in W^{1,\infty}(K_1)$ is
pulled-back to $v(\bfx) := \hv(A^{-1}\bfx)$.  Then, their Lagrange
interpolations $\I_{K_1}^L \hv$ and $\I_{K}^L v$ are defined as
\begin{align*}
   (\I_{K_1}^L \hv)(\widehat{X},\widehat{Y}) & = \widehat{P}
   \widehat{X} + \widehat{R} \widehat{Y} + \widehat{R}, \quad
  \widehat{P} := \frac{\hv(\hat\bfx_2)-\hv(\hat\bfx_1)}
                  {|\hat\bfx_2 - \hat\bfx_1|},    \;
  \widehat{Q} := \frac{\hv(\hat\bfx_3)-\hv(\hat\bfx_1)}
      {|\hat\bfx_3 - \hat\bfx_1|}, \\
   (\I_{K}^L v)(X,Y) & = P X + Q Y + R, \quad
  P = \widehat{P}, \quad 
  Q = -\frac{s}{t}\widehat{P} + \frac{1}{t}\widehat{Q}.
\end{align*}
Therefore, we see that
\begin{align*}
   |\I_{K_1}^L \hv|_{1,\infty,K_1} \le |\hv|_{1,\infty,K_1}, \qquad
   |\I_{K}^L v|_{1,\infty,K} \le \frac{1 + |s|}{t}
   |\I_{K_1}^L \hv|_{1,\infty,K_1}.
\end{align*}
Combining these inequalities with \eqref{transinf}, we have
\begin{align*}
   |\I_{K}^L v|_{1,\infty,K}
   \le \frac{4}{\sin\theta_K} |v|_{1,\infty,K},
\end{align*}
where $\theta_K$ is the maximum angle of $K$.
Note that, in general, the Sobolev norm is affected by a rotation.
Hence, we have shown the following lemma.

\begin{lemma} \label{lem3.1}
Let $\Omega \subset \R^2$ be a bounded polygonal domain and
$\tau$ be a triangulation of $\Omega$.  Suppose that $\tau$
satisfies the maximum angle condition, that is, there exists 
$\theta_2$, $\pi/3 \le \theta_2 < \pi$, such that
$\theta_K \le \theta_2$ for any $K \in \tau$.  Then, 
there exists a constant $C_1$ depending only on $\theta_1$ such that 
\[
   \|\I_\tau^L
  \|_{\mathcal{L}(W^{1,\infty}(\Omega),W^{1,\infty}(\Omega))}
  \le C_1,
\]
where
$\|\I_\tau^L  \|_{\mathcal{L}(W^{1,\infty}(\Omega),W^{1,\infty}(\Omega))}$
is the operator norm of
$\I_\tau^L :W^{1,\infty}(\Omega) \to W^{1,\infty}(\Omega)$.
\end{lemma}

Lemma~\ref{lem3.1} provides an alternate proof of the following
classical result.

\begin{theorem}[Young \cite{Young}] \label{Young}
 Let $\Omega \subset \R^2$ be a bounded domain and
$\{\tau_k\}_{k=1}^\infty$ be a sequence of triangulations
of $\Omega$ with $\lim_{k\to\infty}|\tau_k| = 0$
that satisfies the maximum angle condition.
That is, there exists $\theta_2$, $\pi/3 \le \theta_2 < \pi$, such that
$\theta_K \le \theta_2$ for any $K \in \tau_k$, $k=1,2,\cdots$.
Then, for any $f \in W^{1,\infty}(\Omega)$, we have
\begin{equation}
   \lim_{k\to\infty}A_E(\I_{\tau_k}^L f) = A_L(f).
    \label{limitLagrange}
\end{equation}
\end{theorem}
\begin{proof} First, we note that, for $f$, $g \in W^{1,1}(K)$,
\begin{align}
   \left|\int_K \sqrt{1 + |\nabla f|^2}\, \dd \bfx
    - \int_K \sqrt{1 + |\nabla g|^2}\, \dd \bfx\right|
   \le |f - g|_{1,1,K},
   \label{ineq1}
\end{align}
because
\begin{align*}
   \frac{|\partial^\gamma f + \partial^\gamma g|}
    {\sqrt{1 + |\nabla f|^2} + \sqrt{1 + |\nabla g|^2}} \le 1,
   \quad \gamma = (1,0), \; (0,1).
\end{align*}
Let $\varepsilon > 0$ be arbitrarily taken and fixed.
We may take $f_\varepsilon \in W^{2,\infty}(\Omega)$ such that
$|f - f_\varepsilon|_{1,\infty,\Omega} < \varepsilon$.
Recall that we have the estimation
\[
   |f_\varepsilon - \I_{\tau_k}^L f_\varepsilon|_{1,\infty,\Omega}
   \le C_2 R_k |f_\varepsilon|_{2,\infty,\Omega},
\]
where $R_k := \max_{K \in \tau_k} R_K$ and $C_2$ is a constant
that is independent of $\tau_k$ and $f_\varepsilon$
\cite{KobayashiTsuchiya1}, \cite{KobayashiTsuchiya2},
\cite{KobayashiTsuchiya3}.
If the sequence of triangulations $\{\tau_k\}$ satisfies
the maximum angle condition, then it satisfies the circumradius
condition. Hence, we have $\lim_{k\to\infty} R_k = 0$.

There exists an integer $N$ such that, for any integer $k \ge N$, we have
$C_2 R_k |f_\varepsilon|_{2,\infty,\Omega} < \varepsilon$.
Let $|\Omega|$ be the area of $\Omega$.  It follows from 
Lemma~\ref{lem3.1} and \eqref{ineq1} that, for $k \ge N$,
\begin{align*}
 |A_L(f) - & A_E(\I_{\tau_k}^Lf)| \le 
 |f - \I_{\tau_k}^{L} f|_{1,1,\Omega} \\
 & \le |f - f_\varepsilon|_{1,1,\Omega} 
 + |f_\varepsilon - \I_{\tau_k}^{L} f_\varepsilon|_{1,1,\Omega}
 + |\I_{\tau_k}^{L} (f_\varepsilon - f)|_{1,1,\Omega} \\
 & \le |\Omega|\left(|f - f_\varepsilon|_{1,\infty,\Omega} 
 + |f_\varepsilon - \I_{\tau_k}^{L} f_\varepsilon|_{1,\infty,\Omega}
 + |\I_{\tau_k}^{L} (f_\varepsilon - f)|_{1,\infty,\Omega}\right) \\
 & \le |\Omega|\left(\varepsilon 
   + C_2 R_k |f_\varepsilon|_{2,\infty,\Omega} +
  |\I_{\tau_k}^{L} (f_\varepsilon - f)|_{1,\infty,\Omega}\right) \\
 & \le |\Omega| \left(2\varepsilon + 
  \|\I_{\tau_k}^{L}\|_{\mathcal{L}(W^{1,\infty}(\Omega),W^{1,\infty}(\Omega))} 
    |f - f_\varepsilon|_{1,\infty,\Omega} \right) \\
  & < |\Omega|(2 + C_1) \varepsilon.
\end{align*}
Because $\varepsilon$ is arbitrary, these inequalities indicate
that \eqref{limitLagrange} holds.
\hfill $\square$
\end{proof}

\noindent
\textit{Remark:}
Here, we describe Young's original proof of Theorem~\ref{Young}
concisely.  Let $R := (a,c) \times (b,d)$ be a rectangle.  Let $x(u,v)$
and $y(u,v)$ be sufficiently smooth functions defined on $(u,v) \in R$,
and $B(u,v)$ be defined by
\[
 B(u,v) := \frac{\partial x}{\partial u}
  \frac{\partial y}{\partial v} - \frac{\partial x}{\partial v}
  \frac{\partial y}{\partial u}.
\]
The rectangle $R$ is divided into small rectangles with segments
that are parallel to $u$- and $v$-axes.  As a result, $R$ is divided
into small (possibly very thin) sub-rectangles.
Furthermore, each sub-rectangle is divided into two
semi-rectangles (triangles) by means of the diagonal, sloping down from
left to right.

 Let $h$, $k$ be
sufficiently small reals such that $hk > 0$, and $(u,v)^\top$,
$(u+h,v)^\top$, $(u+h,v+k)^\top$, $(u,v+k)^\top$ be the corner points of
a sub-rectangle.  Define
\begin{align*}
   |D_n| := \frac{1}{2}\bigl|(x(u+h,v) & - x(u,v))(y(u,v+k) - y(u,v)) \\
         & - (y(u+h,v) - y(u,v))(x(u,v+k) - x(u,v))\bigr|
\end{align*}
for one triangle, and also a similar expression for the other triangle.
Young considered $\sum_{n} |D_n|$, where the summation is taken for all
such triangles. He proved that 
\[
  \lim_{\bar{h}, \bar{k} \to 0} \sum_n |D_n| =
    \int_a^b \!\! \int_c^d |B(u,v)| du dv, \quad
    \bar{h} := \max h, \quad \bar{k} := \max k
\]
by rather measure theoretic manner (considering Stieltjes integrals).
The conclusion was immediately extended to the case of surface areas.
Then, he ``skewed'' triangles in sub-rectangles so that one of
angles of every triangle in the $(u,v)$-plane lies between
$0 < \gamma$ and $\pi - \gamma$, and he finally claimed that
Theorem~\ref{Young} is valid.

Therefore, the strategy of his proof was ``compress right triangles
perpendicularly and skew them'', and is similar to ours.

\section{Error analysis of Crouzeix--Raviart interpolation}
\label{sec:CrouzeixRviart}
Let $\gamma \in \N_0^2$ be a multi-index with $|\gamma|=1$.
The sets $\Xi_p^\gamma \subset W^{1,p}(\hK)$, $1 \le p \le \infty$,
are defined by
\begin{align*}
  \Xi_p^{(1,0)} & := \left\{ v \in W^{1,p}(\hK) \Bigm| 
    \int_0^1 v(s,0) \dd s = 0 \right\}, \\
  \Xi_p^{(0,1)} & := \left\{ v \in W^{1,p}(\hK) \Bigm| 
    \int_0^1 v(0,s) \dd s = 0 \right\}.
\end{align*}
Similarly, for an arbitrary triangle $K \subset \R^2$, 
$\E_p(K)$, $\Phi_p(K) \subset W^{1,p}(K)$ are defined by
\begin{align*}
 \Phi_p(K) & := \left\{ v \in W^{1,p}(K) \Bigm| 
    \int_K v \, \dd \bfx = 0 \right\}, \\
  \E_p(K) & := \left\{ v \in W^{1,p}(K) \Bigm| 
    \int_{e_i} v\, \dd s = 0, \; i=1,2,3 \right\}.
\end{align*}
From the definition, it is clear that $\E_p(\hK) \subset \Xi_P^\gamma$,
$|\gamma|=1$.
Then, the constant $A_p$ and $B_p(K)$ are defined for
$p \in [1,\infty]$ by
\begin{gather*}
   A_p := \sup_{v \in \Xi_p^{(1,0)}} \frac{|v|_{0,p,\hK}}
  {|v|_{1,p,\hK}} = \sup_{v \in \Xi_p^{(0,1)}}
   \frac{|v|_{0,p,\hK}}{|v|_{1,p,\hK}}, \\
  B_p(K) := \sup_{v \in \Phi_p(K)} \frac{|v|_{0,p,K}}{|v|_{1,p,K}}, \qquad
  C_p(K) := \sup_{v \in \E_p(K)} \frac{|v|_{0,p,K}}{|v|_{1,p,K}}.
\end{gather*}
The constant $A_p$ is called the
\textbf{Babu\v{s}ka--Aziz constant} for $p$, $1 \le p \le \infty$.
According to Liu--Kikuchi  \cite{LiuKikuchi}, $A_2$ is the maximum
positive solution of the equation $1/x + \tan(1/x) = 0$, and
$A_2 \approx 0.49291$.  Babu\v{s}ka--Aziz \cite{BabuskaAziz} and
Kobayashi--Tsuchiya \cite{KobayashiTsuchiya1} showed the following lemma.

\begin{lemma}\label{L3.1} 
We have $A_p < \infty$, $1 \le p \le \infty$.
\end{lemma}

Similarly, the following lemma holds.

\begin{lemma}\label{L3.2}
We have $B_p(\hK) < \infty$, $1 \le p \le \infty$.
\end{lemma}
\begin{proof} The proof is by contradiction. Suppose that
$B_{p}(\hK) = \infty$.
Then, there exists $\{w_k\}_{k=1}^\infty \subset \Phi_{p}(\hK)$ such that
\[
   |w_k|_{0,p,\hK} = 1, \qquad
   \lim_{k \to \infty} |w_k|_{1,p,\hK} = 0.
\]
By \cite[Theorem~3.1.1]{Ciarlet}, there is a constant $C(\hK,p)$ such that
\begin{align*}
  \inf_{q \in \R} \|v + q\|_{1,p,\hK} & \le 
  C(\hK,p) |v|_{1,p,\hK}, \qquad \forall v \in W^{1,p}(\hK).
\end{align*}
Therefore, there exists $\{q_k\}\subset \R$ such that
\begin{gather*}
   \inf_{q \in \R}
  \|w_k + q\|_{1,p,\hK} \le \|w_k + q_k\|_{1,p,\hK}
   \le \inf_{q \in \R}
  \|w_k + q\|_{1,p,\hK} + \frac{1}{k}, \\
  \lim_{k \to \infty} \|w_k + q_k\|_{1,p,\hK} \le
  \lim_{k \to \infty}
  \left(C(\hK,p)|w_k|_{1,p,\hK} + \frac{1}k \right)= 0. 
\end{gather*}
As the sequence $\{w_k\}\subset W^{1,p}(\hK)$ is bounded,
$\{q_k\} \subset \R$ is also bounded.
Thus, there exists a subsequence $\{q_{k_i}\}$ such that
$q_{k_i}$ converges to $\bar{q} \in \R$.
In particular, we have
\[
   \lim_{k_i \to \infty} \|w_{k_i} + \bar{q}\|_{1,p,\hK} = 0.
\]
Hence, we have
\[
 0 = \lim_{k \to \infty} \int_{\hK} (w_{k_i} + \bar{q})\, \dd \bfx
  = \int_{\hK} \bar{q} \,\dd \bfx,
\]
because $w_{k_i} \in \Phi_{p}(\hK)$. 
Hence, we conclude that $\bar{q} = 0$ and 
$\lim_{k_i \to \infty}\|w_{k_i}\|_{1,p,\hK}$ $ = 0$, which
contradicts $\lim_{k_i \to \infty}|w_{k_i}|_{0,p,\hK}=1$. 
\hfill $\square$
\end{proof}

Let $\alpha \in (0,1]$ and $F_\alpha:\R^2 \to \R^2$ be defined by
$F_\alpha(x,y) := (x,\alpha y)^\top$, $(x,y)^\top \in \R^2$.
An arbitrary $v \in W^{1,p}(K_\alpha)$ is pulled-back to 
$\hv := v\circ F_\alpha \in W^{1,p}(\hK)$.    From the definitions,
it is clear that if $v \in \E_p(\Ka)$ or $v \in \Phi_p(\Ka)$,
then $\hv \in \E_p(\hK)$ or $\hv \in \Phi_p(\hK)$, respectively.
Because
\begin{align*}
   |v|_{0,p,K_\alpha}^p = \alpha |\hv|_{0,p,\hK}^p, \quad
   |v_x|_{0,p,K_\alpha}^p = \alpha |\hv_x|_{0,p,\hK}^p, \quad
   |v_y|_{0,p,K_\alpha}^p = \frac{1}{\alpha^{p-1}}
   |\hv_y|_{0,p,\hK}^p,
\end{align*}
we have, for $v \in W^{1,p}(\Ka)$,
\begin{align*}
   \frac{|v|_{0,p,K_\alpha}^p}{|v|_{1,p,K_\alpha}^p}
   = \frac{|\hv|_{0,p,\hK}^p}
   {|\hv_x|_{0,p,\hK}^p + \frac{1}{\alpha^{p}}|\hv_y|_{0,p,\hK}^p}
    \le \frac{|\hv|_{0,p,\hK}^p}
   {|\hv_x|_{0,p,\hK}^p + |\hv_y|_{0,p,\hK}^p}
  = \frac{|\hv|_{0,p,\hK}^p}{|\hv|_{1,p,\hK}^p}.
\end{align*}
This inequality yields
\begin{gather}
   B_p(\Ka) = \sup_{v \in \Phi_p(\Ka)} \frac{|v|_{0,p,\Ka}}{|v|_{1,p,\Ka}}
\le \sup_{\hv \in \Phi_p(\hK)}\frac{|\hv|_{0,p,\hK}^p}
  {|\hv|_{1,p,\hK}^p} = B_p(\hK) < \infty, 
   \label{B-est} \\
 C_p(\Ka) = \sup_{v \in \E_p(K_\alpha)} 
   \frac{|v|_{0,p,K_\alpha}}{|v|_{1,p,K_\alpha}}
  \le \sup_{\hv \in \E_p(\hK)} \frac{|\hv|_{0,p,\hK}}{|\hv|_{1,p,\hK}}
  \le \sup_{\hv \in \Xi_p^{(1,0)}} \frac{|\hv|_{0,p,\hK}^p}
  {|\hv|_{1,p,\hK}^p} = A_p < \infty. \label{C-est}
\end{gather}

Recall that $\tK$ defined in Section~2.4 and depicted in
Figure~\ref{LiuKikuchiTriangle} is the triangle with vertices
$(0,0)^\top$, $(0,1)^\top$, $(\alpha s, \alpha t)^\top$, where
$0 < \alpha \le 1$, $s^2 + t^2 = 1$, and $t > 0$.  Using the
inequalities in Section~2.4, we find that
\begin{align}
     \frac{|v|_{0,p,\tK}}{|v|_{1,p,\tK}} & \le
     \frac{2^{\eta(p)}t|\hv|_{0,p,\Ka}}
     {(1-|s|)^{1/2}|\hv|_{1,p,\Ka}}
  \le 2\frac{|\hv|_{0,p,\Ka}}{|\hv|_{1,p,\Ka}},
   \qquad \forall v \in W^{1,p}(\tK),
   \label{trans1}
\end{align}
because
\[
  2^{\eta(p)}(1+|s|)^{1/2}\le 
   \begin{cases}
         2^{1/p}, &  1 \le p \le 2 \\
         2^{1 - 1/p}, & 2 \le p < \infty
     \end{cases}.
\]
This estimation \eqref{trans1} with \eqref{B-est} and \eqref{C-est} yields
\begin{align*}
   B_p(\tK) & := \sup_{v \in \Phi_p(\tK)}
   \frac{|v|_{0,p,\tK}}{|v|_{1,p,\tK}} \le 2
     \sup_{\hv \in \Phi_p(\Ka)}
   \frac{|\hv|_{0,p,\Ka}}{|\hv|_{1,p,\Ka}} \le 2 B_p(\hK), \\
   C_p(\tK) & := \sup_{v \in \E_p(\tK)}
   \frac{|v|_{0,p,\tK}}{|v|_{1,p,\tK}} \le 2
     \sup_{\hv \in \E_p(\Ka)}
   \frac{|\hv|_{0,p,\Ka}}{|\hv|_{1,p,\Ka}} \le 2 A_p.
\end{align*}

The above estimations can be extended to general triangles. Now, let $K$
be an arbitrary triangle.  The similar transformation $G_\beta:\R^2 \to \R^2$
for a positive $\beta \in \R$ is defined by
$G_\beta({\bf x}) := \beta {\bf x}$.  Let $K_1$ be defined by
$K_1 = G_\beta(K)$.  A function $u \in W^{k,p}(K)$ on $K$ is pulled-back
to $v(\bfx) := u(G_\beta^{-1}(\bfx)) = u(G_{1/\beta}(\bfx))$ on
$K_1$.  Then, for a nonnegative integer $k$ and any
$p$ $(1 \le p \le \infty)$, we have
\[
  |v|_{k,p,K_1} = \beta^{2/p-k}|u|_{k,p,K}, \qquad
  \forall u \in W^{p,k}(K).
\]

Let $h_K \ge h_1 \ge h_2$ be the lengths of the three edges of $K$.
Suppose that the second longest edge of $K$ is parallel to
the $x$- or $y$-axis.  Then, by a combination of translation,
 mirror imaging, and $G_{1/h_1}$, $K$ can be transformed to the
triangle $\tK$.  Hence, we may apply the above estimations to
$\widetilde{K}$ to obtain
\begin{gather*}
    \sup_{u\in \Phi_p(K)} 
    \frac{|u|_{0,p,K}}{h_1|u|_{1,p,K}} =
    \sup_{v\in \Phi_p(\tK)} 
    \frac{|v|_{0,p,\tK}}{|v|_{1,p,\tK}} \le \, 2B_p(\hK), \\
        \sup_{u\in \E_p(K)} 
    \frac{|u|_{0,p,K}}{h_1|u|_{1,p,K}} =
    \sup_{v\in \E_p(\tK)} 
    \frac{|v|_{0,p,\tK}}{|v|_{1,p,\tK}} \le \, 2A_p
\end{gather*}
and
\begin{gather*}
    \sup_{u\in \Phi_p(K)}\frac{|u|_{0,p,K}}{|u|_{1,p,K}}
      \le \, 2 B_p(\hK) h_K, \qquad 
    \sup_{u\in \E_p(K)}\frac{|u|_{0,p,K}}{|u|_{1,p,K}}
      \le \, 2 A_ph_K.
\end{gather*}
Note that if $p \neq 2$, the Sobolev norms are affected by a rotation.
Therefore, we have obtained the following theorem.

\begin{theorem} \label{Thm0-1}
Let $K$ be an arbitrary triangle and $h_K:=\mathrm{diam}K$.  There exists a
constant $C=C(p)$ depending only on $p$ such that
\begin{equation*}
  \sup_{u\in \Phi_p(K)}\frac{|u|_{0,p,K}}{|u|_{1,p,K}}
      \le \, C h_K, \quad 
    \sup_{u\in \E_p(K)}\frac{|u|_{0,p,K}}{|u|_{1,p,K}}
      \le \, C h_K, \quad 1 \le p \le \infty.
\end{equation*}
\end{theorem}
An important point in Theorem~\ref{Thm0-1} is that the constant $C$
is independent of the geometry of $K$.

For $f \in L^1(K)$, we define $\bar f \in \R$ by
\[
    \bar f := \frac{1}{|K|}\int_K f(\bfx) \dd \bfx.
\]
From this definition, it is clear that, for arbitrary $f \in L^p(K)$,
\begin{equation}
    \int_K (f - \bar f) \dd \bfx = 0 \quad \text{ and } \quad
     |\bar f|_{0,p,K} \le |f|_{0,p,K}.
   \label{basic-property}
\end{equation}
Hence, we may apply Theorem~\ref{Thm0-1} to obtain
the Poincar\'e--Wirtinger inequality for triangles.

\begin{corollary}[Poincar\'e--Wirtinger inequality] \label{CorPW}
Let $K$ be an arbitrary triangle.
Then, for $p$, $1 \le p \le \infty$ and the constant $C = C(p)$
that appeared in Theorem~\ref{Thm0-1}, the following estimation holds:
\begin{equation*}
    |f - \bar f|_{0,p,K} \le C h_K |f|_{1,p,K}, 
       \qquad \forall f \in W^{1,p}(K).
\end{equation*}
\end{corollary}

\noindent
\textit{Remark:} The Poincar\'e--Wirtinger inequality is standard and
mentioned in many textbooks.  However, the inequality is generally shown 
under conditions on the domains.  For example, it is stated
in \cite{Brezis} with the condition that the domain is
of $C^1$ class.  In \cite{GilbargTrudinger}, the inequality (7.45)
on page 164 can be read as
\[
    |f - \bar f|_{0,p,\Omega} \le
    \left(\frac{\omega_d}{|\Omega|}\right)^{1-1/d}(\mathrm{diam}\Omega)^d
     |f|_{1,p,\Omega}, \qquad \forall f \in W^{1,p}(\Omega),
\]
where $\Omega \subset \R^d$ is a bounded convex domain and
$\omega_d$ is the $(d-1)$-dimensional Hausdorff measure of the unit sphere
$S^{d-1} \subset \R^d$.  Note that if $\Omega$ becomes very ``flat'',
then the coefficient on the right-hand side diverges.
For cases of degenerate (``flat'') domains, Payne--Weinberger
\cite{PayneWeinberger} and Laugesen--Siudeja \cite{LaugesenSiudeja} gave
estimations for the case $p=2$.  Thus, Corollary~\ref{CorPW} is an
extension of prior results.

\vspace{0.4cm}
Because of \eqref{basic-property}, the following lemma obviously holds.
\begin{lemma}\label{Lem0-0}
For any $f \in L^p(K)$, $1 \le p \le \infty$, we have
\begin{equation*}
  |f - \bar{f}|_{0,p,K} \le 2 |f|_{0,p,K}.
\end{equation*}
\end{lemma}

\vspace{0.4cm}
We now consider error estimates of the Crouzeix--Raviart interpolation
$\Icr v$.  Let $K$ be an arbitrary triangle and $1 \le p \le \infty$.
From the definition of $\Icr v$ and the divergence theorem, we notice that
\begin{align*}
   \int_{K} \left(v - \Icr v\right)_x \dd \bfx =
  \int_{\partial K} \left(v - \Icr v\right) n_1 \dd s
   = \sum_{i=1}^3 n_1 \int_{e_i} \left(v - \Icr v\right) \dd s = 0,
\end{align*}
where $\mathbf{n} = (n_1,n_2)^\top$ is the outer unit normal vector
on $\partial K$, which is a constant vector on each edge.
Similarly, we have
\begin{align*}
    \int_{K} \left(v - \Icr v\right)_y \dd \bfx = 0.
\end{align*}
Because $\Icr v \in \PP_1$ and $\left(\Icr v\right)_x$,
$\left(\Icr v\right)_y$ are constants on $K$, these equalities imply that
\begin{gather*}
   \left(\Icr v\right)_x = \frac{1}{|K|} \int_K v_x \dd \bfx
    =: \overline{v_x}, \qquad
    \left(\Icr v\right)_y = \frac{1}{|K|} \int_K v_y \dd \bfx
    =: \overline{v_y}, \\
   \Icr v (x,y) = (\overline{v_x}) x + (\overline{v_y}) y + c,
  \qquad c \in \R.
\end{gather*}
Therefore, \eqref{basic-property} and Poincar\'e--Wirtinger inequality
yield, for arbitrary $v \in W^{2,p}(K)$,
\begin{align}
 \left|v - \Icr v\right|_{1,p,K} \le C h_K |v|_{2,p,K} \quad \text{ and }
  \quad \left|\Icr v\right|_{1,p,K} \le |v|_{1,p,K}.
  \label{est-12}
\end{align}

Note that $v - \Icr v \in \E_p(K)$ for any $v \in W^{1,p}(K)$.
Thus, Theorem~\ref{Thm0-1} and \eqref{est-12} imply that
\begin{align}
   |v - \Icr v|_{0,p,K} \le C h_K |v - \Icr v|_{1,p,K} \le C^2h_K^2
 |v|_{2,p,K} \quad \forall v \in W^{2,p}(K).
  \label{est-012}
\end{align}
Moreover, it follows from Lemma~\ref{Lem0-0} that, for $1 \le p < \infty$,
\begin{align*}
 |v - \Icr v|_{1,p,K}^p & = |v_x - \left(\Icr v\right)_x|_{0,p,K}^p
  + |v_y - \left(\Icr v\right)_y|_{0,p,K}^p \\
  & = |v_x - \overline{v_x}|_{0,p,K}^p
  + |v_y - \overline{v_y}|_{0,p,K}^p \\
  & \le 2^p\left(|v_x|_{0,p,K}^p + |v_y|_{0,p,K}^p\right)  
   = 2^p |v|_{1,p,K}^p.
\end{align*}
The case of $p = \infty$ is similar.  Hence, we obtain
\begin{align}
   |v - \Icr v|_{0,p,K} \le 2C h_K |v|_{1,p,K},
  \qquad \forall v \in W^{1,p}(K).
  \label{est-01}
\end{align}

Gathering estimates \eqref{est-012} and \eqref{est-01}, we obtain
the following theorem.
\begin{theorem}\label{thm-main2}
Let $K$ be an arbitrary triangle and $1 \le p \le \infty$.
Then, for the Crouzeix--Raviart interpolation $\Icr v$, the following
error estimations hold:
\begin{gather*}
     |v - \Icr v|_{0,p,K} \le 2 C h_K |v|_{1,p,K},
  \qquad \forall v \in W^{1,p}(K), \\
   |v - \Icr v|_{0,p,K} \le C h_K |v - \Icr v|_{1,p,K} \le C^2h_K^2
 |v|_{2,p,K}, \quad \forall v \in W^{2,p}(K).
\end{gather*}
Here, the constant $C=C(p)$ is from Theorem~\ref{Thm0-1} and is
independent of the geometry of $K$.
\end{theorem}

\section{Approximating the surface area by Crouzeix--Raviart interpolation}
\label{main-section}
Recall that $\Omega \subset \R^2$ is a polygonal domain and
$\{\tau_k\}_{k=1}^\infty$ is a sequence of triangulations
of $\Omega$ with $\lim_{k \to \infty} |\tau_k| = 0$.
Let $f \in W^{1,\infty}(\Omega)$.
The surface area $A_L(f)$ in the sense of Lebesgue is approximated by
Crouzeix--Raviart interpolation as
\begin{align*}
   A_{\tau_k}^{CR}(f) := \sum_{K \in \tau_k} \int_K
   \sqrt{1 + |\nabla (\I_K^{CR} f)|^2} \, \dd \bfx.
\end{align*}

Let $\varepsilon > 0$ be arbitrarily taken and fixed.
We may take $f_\varepsilon \in W^{2,1}(\Omega)$ such that
$|f - f_\varepsilon|_{1,1,\Omega} < \varepsilon$.
There exists an integer $N$ such that, for any integer
$k \ge N$, we have
$C |\tau_k| |f_\varepsilon|_{2,1,\Omega} < \varepsilon$,
where the constant $C$ is from Theorem~\ref{thm-main2}.
It follows from \eqref{ineq1} and \eqref{est-12} that, for $k \ge N$,
\begin{align*}
 |A_L(f) - & A_{\tau_k}^{CR}(f)| \le \sum_{K \in \tau_k}
 |f - \I_K^{CR} f|_{1,1,K} \\
 & \le \sum_{K \in \tau_k} \left(
   |f - f_\varepsilon|_{1,1,K} 
 + |f_\varepsilon - \I_K^{CR} f_\varepsilon|_{1,1,K}
 + |\I_K^{CR} (f_\varepsilon - f)|_{1,1,K}\right) \\
 & \le 2|f - f_\varepsilon|_{1,1,\Omega}
   + C |\tau_k| |f_\varepsilon|_{2,1,\Omega} < 3 \varepsilon.
\end{align*}
Therefore, we have shown the following theorem.

\begin{theorem}\label{thm:main}
Let $\Omega \subset \R^2$ be a bounded polygonal domain
and $\{\tau_k\}_{k=1}^\infty$ be a sequence of triangulations
of $\Omega$ such that $\lim_{k \to \infty} |\tau_k| = 0$.
Let $f \in W^{1,\infty}(\Omega)$ and $A_{\tau_k}^{CR}(f)$ be the
approximation of the surface area $A_L(f)$ by Crouzeix--Raviart
interpolation.  Then, we have
\begin{equation}
   \lim_{k \to \infty} A_{\tau_k}^{CR}(f) = A_L(f).
   \label{convCR}
\end{equation}
\end{theorem}

It is clear from the proof that
\eqref{convCR} holds under the assumptions that
$f \in W^{1,1}(\Omega) \cap C^0(\overline{\Omega})$ with
$A_L(f) < \infty$ and $f$ is absolutely continuous in the sense of
Tonelli.  We here strongly emphasize that \textit{we have not imposed
any geometric conditions on $\{\tau_k\}$}, such as the maximum angle
condition, or the circumradius condition.

\section{Approximating areas of surfaces in parametric form}
\label{sec:parametric}
In this section, we show that the results obtained so far can be
straightforwardly extended to the case of parametric surfaces.
Let $\Omega \subset \R^2$ be a bounded polygonal domain and
$\bff : \Omega \to \R^3 \in W^{1,\infty}(\Omega;\R^3)$.
Because the Jacobian matrix
\begin{equation*}
  D\bff(\bfx) :=
\begin{pmatrix}
   \frac{\partial f^1}{\partial x} & \frac{\partial f^1}{\partial y} \\
   \frac{\partial f^2}{\partial x} & \frac{\partial f^2}{\partial y} \\
   \frac{\partial f^3}{\partial x} & \frac{\partial f^3}{\partial y} \\
 \end{pmatrix}
\end{equation*}
exists almost everywhere in $\Omega$, we may assume that 
$\mathrm{rank}Df(\bfx) = 2$ almost everywhere in $\Omega$.  Then, the
image of $\bff$ is a surface (possibly with self-intersections) in
$\R^3$.  Its area $A_L(\bff)$ in the sense of Lebesgue is defined as
before (see \cite{Saks} for details), and is equal to
\begin{equation*}
   A_L(\bff) = \int_\Omega |\bff_x \times \bff_y| \dd\bfx,
\end{equation*} 
where  
\[
   \bff_x := \left(\frac{\partial f^1}{\partial x},
      \frac{\partial f^2}{\partial x}, \frac{\partial f^3}{\partial x}
        \right)^\top, \qquad
      \bff_y := \left(\frac{\partial f^1}{\partial y},
      \frac{\partial f^2}{\partial y}, \frac{\partial f^3}{\partial y}
        \right)^\top
\]
and $\bff_x \times \bff_y$ is the exterior product of $\bff_x$,
$\bff_y$.  The surface area $A_L(\bff)$ can be discussed
in terms of the Hausdorff measure and the area formula.  See
\cite[Chapter~4]{EvansGariepy} for details.

We now consider interpolations of $\bff$.  Let
$\{\tau_n\}_{n=1}^\infty$ be a sequence of triangulations of $\Omega$.
On each $\tau_n$, the Lagrange and Crouzeix--Raviart interpolations
$\I_{K}^L\bff$, $\I_{K}^{CR}\bff$ are defined component-wise.
Then, $A_L(\bff)$ is approximated by $A_L(\I_{\tau_n}^L\bff)$
and
\[
   A_{\tau_n}^{CR}(\bff) := \sum_{K \in \tau_n}
   \int_K \left|\left(\I_K^{CR}\bff\right)_x \times
 \left(\I_K^{CR}\bff\right)_y\right| \dd\bfx,
\]
respectively.

To simplify the notation, we introduce the vectors
$\F = (F_1,F_2,F_3)^\top$ and $\G = (G_1,G_2,G_3)^\top$
defined by
\begin{align*}
 \F := \bff_x \times \bff_y, \qquad
 \G := \bfg_x \times \bfg_y,
\end{align*}
where 
$\bfg := \I_{K}^L\bff$ or $\bfg := \I_{K}^{CR} \bff$.  Then, the error
$\left|A_L(\bff) - A_L(\I_{\tau_n}^L\bff)\right|$ is estimated as
\begin{align*}
   \left|A_L(\bff) - A_L(\I_{\tau_n}^L\bff)\right| & =
    \left|\int_\Omega |\F| \dd \bfx - 
          \int_\Omega |\G| \dd \bfx \right| \\
  & \le \sum_{i=1}^3 \int_\Omega 
  \frac{|F_i +G_i||F_i - G_i|}{|\F| + |\G|} \dd \bfx \\
  & \le \sum_{K \in \tau_n}\sum_{i=1}^3 \int_K |F_i - G_i| \dd \bfx,
\end{align*}
because $|F_i+G_i|/(|\F| + |\G|) \le 1$.
The error $\left|A_L(\bff) - A_{\tau_n}^{CR}(\bff)\right|$ is 
estimated in a similar manner.

Note that $F_i$ and $G_i$ are written as
\begin{align*}
   F_i = f_x^k f_y^l - f_y^k f_x^l, \qquad
   G_i = g_x^k g_y^l - g_y^k g_x^l, \quad k,l= 1,2,3, \; k \neq l,
\end{align*}
where $g^i = \I_{K}^Lf^i$ or $g^i = \I_{K}^{CR}f^i$.
Therefore, we see that
\begin{align*}
 |F_i - G_i| & \le |f_x^k - g_x^k||f_y^l| 
   + |g_x^k||f_y^l - g_y^l| + |f_y^k - g_y^k| |f_x^l|
   + |g_y^k||f_x^l - g_x^l|
\end{align*}
and
\begin{align*}
   \sum_{i=1}^3 \int_K |F_i - G_i| \dd \bfx
  \le  \left(|\bff|_{1,\infty,K} + |\bfg|_{1,\infty,K}\right)
            |\bff - \bfg|_{1,1,K}.
\end{align*}

For the case of Lagrange interpolation, $\bfg = \I_K^L\bff$ and
we assume that a sequence of triangulations $\{\tau_n\}$ of $\Omega$
satisfies the maximum angle condition.  Then, by
Lemma~\ref{lem3.1}, there exists a constant $C_1$ such that
$|\I_K^L\bff|_{1,\infty,K} \le C_1|\bff|_{1,\infty,K}$, where
the constant $C_1$ depends on the maximum angle.  Thus, we have
\begin{align*}
  \left|A_L(\bff) - A_L(\I_{\tau_n}^L\bff)\right| & =
  (1+C_1) \sum_{K \in \tau_n} |\bff|_{1,\infty,K}
   |\bff - \bfg|_{1,1,K} \\
  & \le (1+C_1) |\bff|_{1,\infty,\Omega}
   \left|\bff - \I_{\tau_n}^L \bff\right|_{1,1,\Omega}.
\end{align*}

Similarly, for the case of Crouzeix--Raviart interpolation, we
have $\bfg = \I_K^{CR}\bff$ and
\begin{align*}
  \left|A_L(\bff) - A_\tau^{CR}(\bff)\right| \le
  2 \sum_{K \in \tau_n} |\bff|_{1,\infty,K}
   \left|\bff - \I_{K}^{CR}\bff\right|_{1,1,K}
\end{align*}
without any geometric condition on the triangulations.

From these inequalities, the following theorem can be shown in
exactly the same manner as used in Sections~3 and 5.

\begin{theorem}
Let $\Omega \subset \R^2$ be a bounded polygonal domain and
$\{\tau_k\}$ be a sequence of triangulations of $\Omega$.
Let $\bff:\Omega \to \R^3$ belong to $W^{1,\infty}(\Omega;\R^3)$ and
$\mathrm{rank}D\bff(\bfx) = 2$ almost everywhere in $\Omega$.

If $\{\tau_k\}$ satisfies the maximum angle condition,
we have the convergence
\begin{equation*}
   \lim_{k \to \infty} A_L(\I_{\tau_k}^L \bff) = A_L(\bff)
\end{equation*}
for Lagrange interpolation. 
Furthermore, we have 
\begin{equation*}
   \lim_{k \to \infty} A_{\tau_k}^{CR}(\bff) = A_L(\bff)
\end{equation*}
for Crouzeix--Raviart interpolation without
any geometric condition on the triangulation $\{\tau_k\}$.
\end{theorem}

\section{Numerical experiments and concluding remarks}
\label{numericalexpriments}
To confirm the results obtained in this paper, we conducted
numerical experiments.  Let $\Omega := (-1,1) \times (-1,1)$ and $N$ be
a positive integer.  We use the triangulation $\tau$ that consists of
congruent isosceles triangles with base length $h:=2/N$ and height
$2/\lfloor2/h^\alpha \rfloor \approx h^\alpha$, $\alpha > 1$.  Note
that the circumradius of the triangle is approximately equal to 
$h^\alpha/2 + h^{2-\alpha}/8$.  Thus, it diverges when $\alpha > 2$ as
$N \to \infty$.  The triangulation of $\Omega$ with $N=12$ and
$\alpha = 1.6$ is shown in Figure~4.

\vspace{0.5cm}
\begin{center}
\begin{minipage}[c]{5.6cm}
  \includegraphics[width=5.5cm]{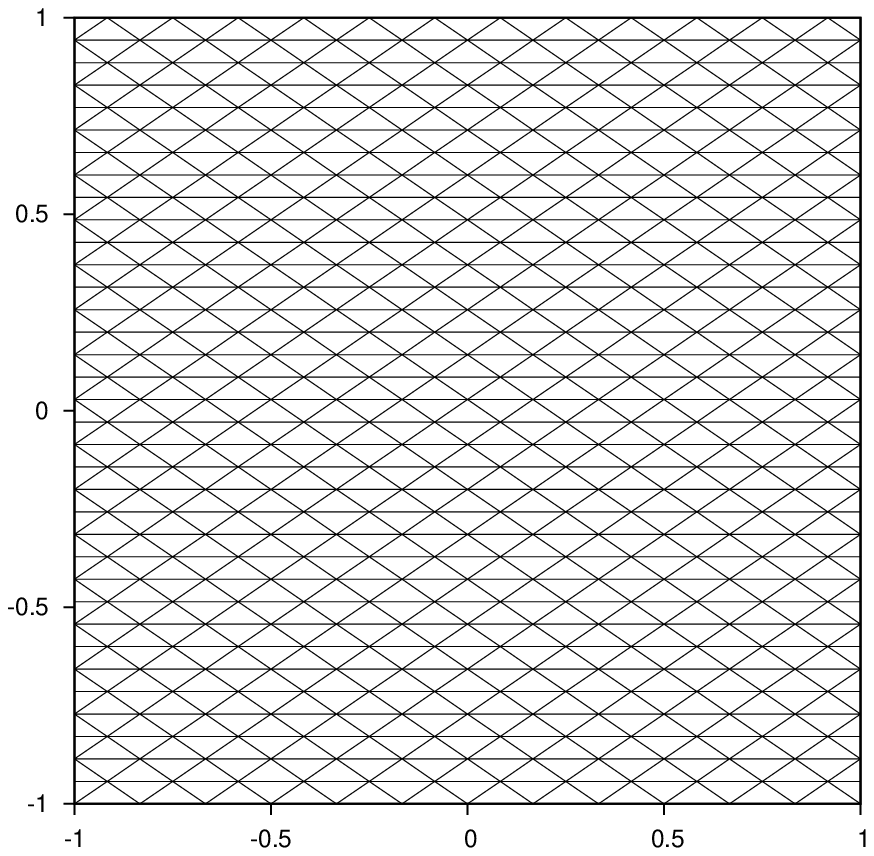}
 \end{minipage}
 \begin{minipage}[c]{5.6cm}

 {\small \textbf{Fig.~4} The triangulation of $\Omega$ with $N=12$ and $\alpha=1.6$.}
 \end{minipage}      
\end{center}

\setcounter{figure}{4}
Let $f(x,y) := (a^2 - x^2)^{1/2}$ with $a = 1.1$.  We computed
$|A_L(f) - A_E(\I_\tau^L f)|$ and $|A_L(f) - A_\tau^{CR}(f)|$
with various $N$ and $\alpha$.  The results are shown in Figure~\ref{fig5}.
\begin{figure}[htbp]
\begin{center}
      \includegraphics[width=8cm]{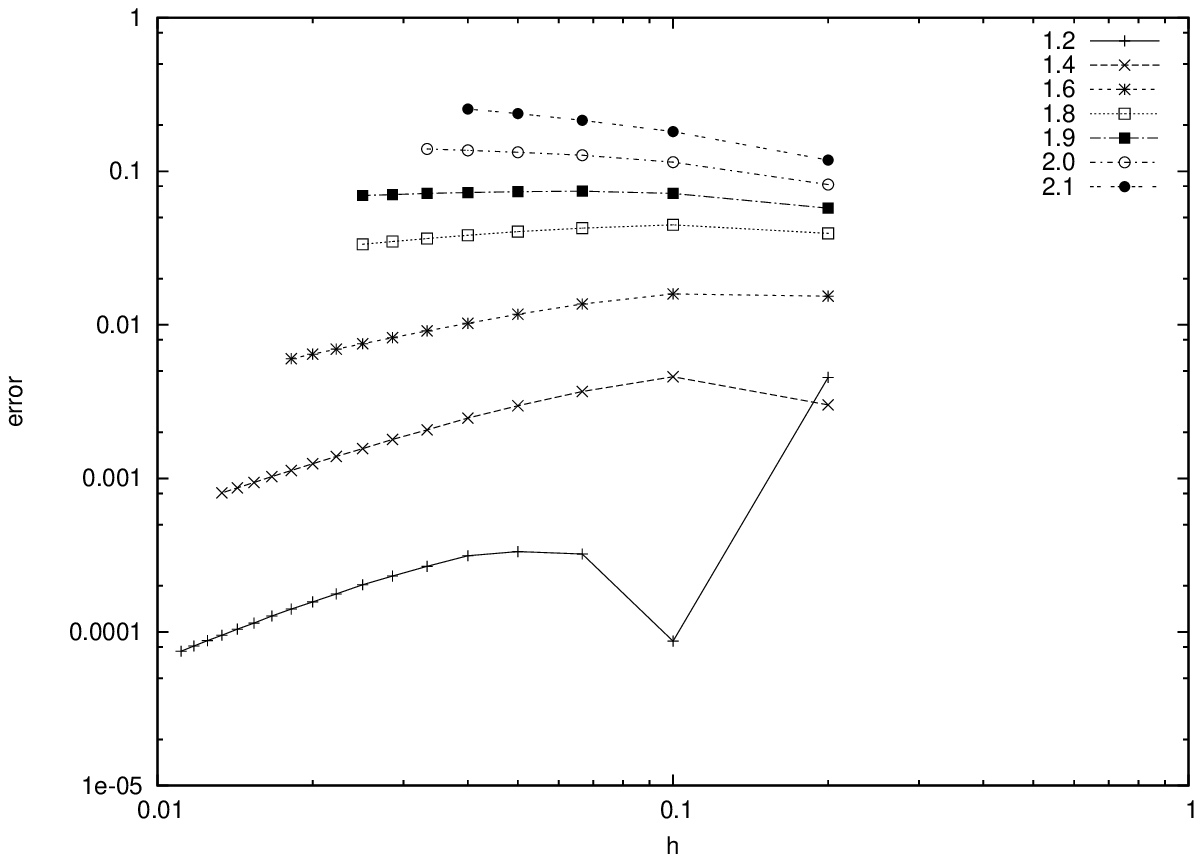} \\
      \includegraphics[width=8cm]{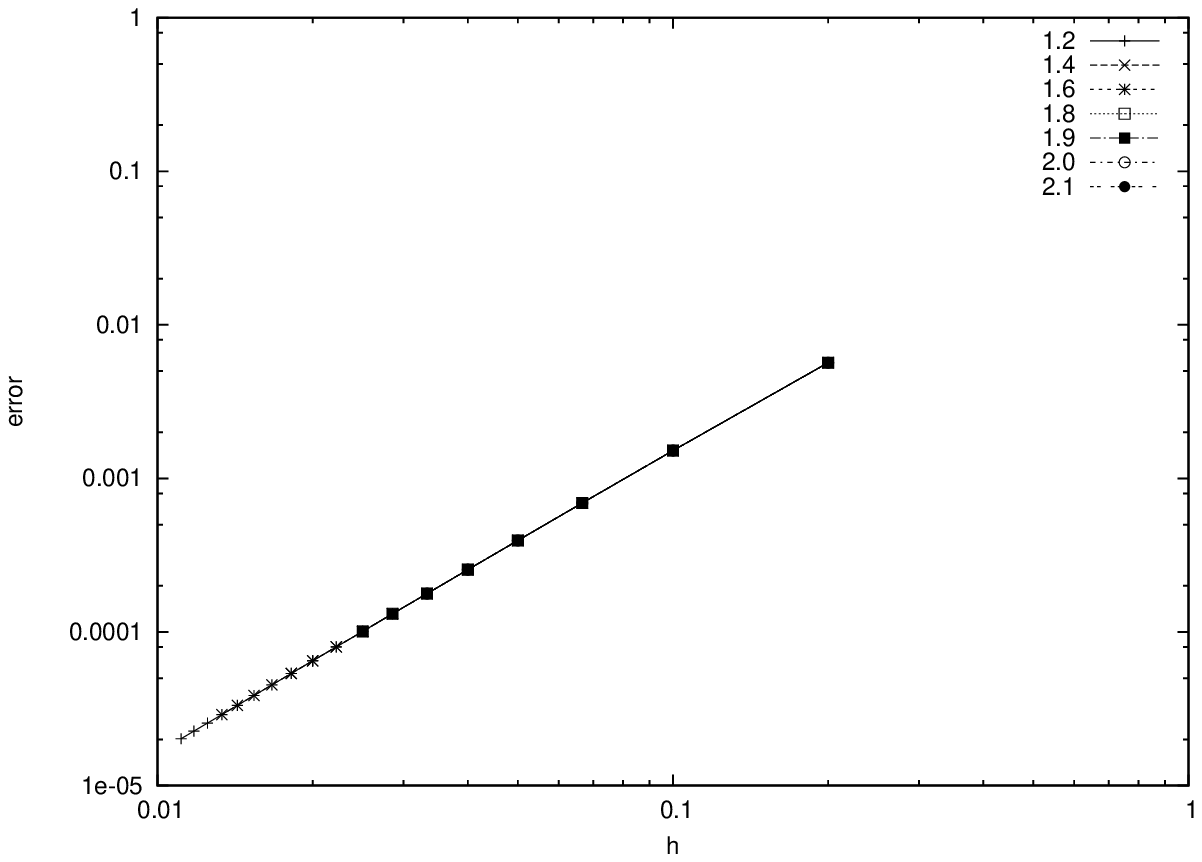}
\end{center}
 \caption{The errors of Lagrange (upper) and Crouzeix--Raviart
 (lower) interpolations.  The \textit{number} next to the \textit{symbol}
 indicates the value of $\alpha$.  The horizontal axis represents the
maximum size of triangles and the vertical axis represents
the errors $|A_E(\I_{\tau_k}^L f) - A_L(f)|$ (upper) and
 $|A_{\tau_k}^{CR}(f) - A_L(f)|$ (lower). }
\label{fig5}
\end{figure}

Note that, as predicted by the error estimations obtained in this
paper, the behavior of the error $|A_L(f) - A_\tau^{CR}(f)|$ does not
depend on $\alpha$ (all the curves overlap and look like just
one curve), whereas the error $|A_L(f) - A_E(\I_\tau^L f)|$
behaves differently as $\alpha$ varies.  We can also see that, when $N$
is small, the errors in the Lagrange interpolation behave strangely for
some reason that the authors cannot explain.

We obtained an alternative proof of the classical result by Young
(Theorem~\ref{Young}).  That is, we have shown that the areas of
the Lagrange interpolation of a surface (of class $W^{1,\infty}$) converge
to the area of the surface under the maximum angle condition on
the triangulation.  The authors conjecture that the same result holds under
the circumradius condition.
Moreover, we showed that the areas of the Crouzeix--Raviart interpolation
of a surface (of class $W^{1,\infty}$) converge to the area of the
surface without any geometric condition on the triangulation.

The authors believe that the results of this paper provide a new insight on
the definition of surface area and related subjects.  In the following,
we mention some immediate problems that arise from this study.
\begin{itemize}
 \item Prove or disprove the conjecture that
Theorem~\ref{Young} holds under the circumradius condition on
triangulations.
\item  The surface area in the sense of Lebesgue is defined using Lagrange
interpolation (or using the subspace $S_{\tau_n}$).  Can we give an
alternate definition of surface area using Crouzeix--Raviart
interpolation (or using a corresponding finite dimensional space) that is
equivalent to the original definition?
\item All the results in this paper are proved under the assumption
$A_L(f) < \infty$.  Let $f \in C^0(\Omega)$ and $\{\tau_k\}$ be a
sequence of triangulations such that
$\lim_{k \to \infty} |\tau_k| = 0$.  In this case, the Crouzeix--Raviart
interpolation $\I_{\tau_k}^{CR}f$ is well-defined.  Suppose that
$\limsup_{k \to \infty} A_{\tau_k}^{CR}(f) < \infty$.
Then, can we show that $A_L(f) < \infty$?  If not, give a
counter-example. 
\item Extend Theorem~\ref{thm:main} to the case of the volume
of the graph of a function with $d$ variables, $d \ge 3$.
\end{itemize}

The authors hope this paper will inspire further research and that one
or more of the above-mentioned questions will be solved in the near future.

\vspace{0.3cm}
\textsc{Acknowledgements.}
The authors were supported by JSPS KAKENHI Grant Numbers
JP16H03950, JP25400198, and JP26400201.

\end{document}